\documentclass[11pt, reqno]{amsart}

\usepackage[english]{babel}
\usepackage{amssymb,amsmath,amsthm}
\usepackage{color}
\usepackage{url}
\usepackage[backref]{hyperref}
\usepackage[alphabetic,backrefs,lite]{amsrefs}
\usepackage{graphicx}

\DeclareFontFamily{U}{wncy}{}
\DeclareFontShape{U}{wncy}{m}{n}{<->wncyr10}{}
\DeclareSymbolFont{mcy}{U}{wncy}{m}{n}
\DeclareMathSymbol{\Sha}{\mathord}{mcy}{"58}

\newtheorem{thm}{Theorem}[section]
\newtheorem{lemma}[thm]{Lemma}
\newtheorem{prop}[thm]{Proposition}

\theoremstyle{definition}
\newtheorem{defn}[thm]{Definition}

\newtheorem{remark}[thm]{Remark}

\newtheorem{assumption}[thm]{Assumption}
\newtheorem{notation}[thm]{Notation}

\newcommand{\bbC}{{\mathbb C}}
\newcommand{\bbF}{{\mathbb F}}
\newcommand{\bbQ}{{\mathbb Q}}
\newcommand{\bbR}{{\mathbb R}}
\newcommand{\bbZ}{{\mathbb Z}}
\newcommand{\calO}{{\mathcal O}}

\DeclareMathOperator{\Aut}{Aut}
\DeclareMathOperator{\Br}{Br}
\DeclareMathOperator{\Gal}{Gal}
\DeclareMathOperator{\GL}{GL}
\DeclareMathOperator{\HH}{H}
\DeclareMathOperator{\inv}{inv}
\DeclareMathOperator{\Norm}{Norm}
\DeclareMathOperator{\Sel}{Sel}
\DeclareMathOperator{\rk}{rk}

\newcommand{\eps}{\varepsilon}
\newcommand{\ol}{\overline}

\title{Calculating the Tate local pairing for any odd prime number}
\author{Erik Visse}
\thanks{I am grateful to Rachel Newton and David Holmes for supervising the writing of my master's thesis, to Hendrik Lenstra and Ronald van Luijk for their helpful comments on said thesis, and again to Rachel Newton for proofreading an earlier draft of this paper.}

\address{Mathematisch Instituut, Universiteit Leiden, Niels Bohrweg 1, 2333CA, Leiden, the Netherlands}
\email{h.d.visse@math.leidenuniv.nl}
\urladdr{http://pub.math.leidenuniv.nl/~vissehd/}

\date{\today}

\begin{document}

\begin{abstract}
Fisher and Newton have given an explicit description of the Tate local pairing associated with the 3-torsion of an elliptic curve. The present paper summarizes the work from the author's master's thesis~\cite{mscthesis} and gives an explicit formula for any odd prime $p$, thereby extending the work of Fisher and Newton.
\end{abstract}

\maketitle

\section{Introduction}

Throughout this paper, let $(E,\calO)$ be an elliptic curve over a number field $k$. The Cassels--Tate pairing associated with $E$ is an alternating pairing
\begin{displaymath}
\langle\ ,\ \rangle:\Sha(E/k)\times \Sha(E/k)\to \bbQ/\bbZ.
\end{displaymath}
There are several equivalent definitions for this pairing, which can for example be found in~\cite{Poonen-Stoll}*{Section 3}. Given as a pairing on $\Sha(E/k)\times\Sha(E/k)$, for any prime number $p$, the Cassels--Tate pairing extends a pairing on $\Sel^{p}(E/k)\times \Sel^{p}(E/k)$ having a kernel that contains $E(k)/pE(k)$.

Cassels introduced this pairing in \cite{Cassels-original} and in the same paper studied its kernel. In~\cite{Cassels-p=2}, he proved that for $p=2$ the pairing can be written as a sum of local pairings. 

One main result of Fisher and Newton~\cite{Fisher-Newton}*{Theorem 1.3} allows one to calculate a restriction of the pairing $\langle\ ,\ \rangle$ as an infinite sum of local pairings $[\ ,\ ]_v$, one for each place $v$ of $k$. The restriction is to one side of the domain being $p$-torsion where $p$ is an odd prime number, so it is useful in calculating the pairing on $\Sel^p(E/k)\times\Sel^p(E/k)$. The symbol $[\ ,\ ]_v$ is called the Tate local pairing and is defined in the next section. 

Both Cassels' method and that of Fisher and Newton involve the Hilbert norm residue symbol. The latter method is necessarily more involved since the Cassels--Tate pairing is alternating and the Hilbert norm residue symbol is symmetric. Only for $p=2$ do these two properties coincide. 

The other main result of Fisher and Newton~\cite{Fisher-Newton}*{Theorem 2.1} deals with their local computation. Part of the theorem gives an explicit element of a local field that is needed to run the calculation, but only for $p=3$. The present paper extends their method to give an explicit element for an odd prime number $p$.

\subsection{Applications}

The pairing on $\Sel^p(E/k)\times \Sel^p(E/k)$ having a one-sided kernel $N$ that contains the subgroup $E(k)/pE(k)$ allows an application in bounding the Mordell--Weil rank of $E(k)$ from the inclusions
\begin{displaymath}
E(k)/pE(k)\subseteq N \subseteq \Sel^p(E/k)
\end{displaymath}
and hence from the inequalities
\begin{displaymath}
\#(E[p](k))\cdot p^{\rk E(k)} \leq \#N \leq \#\Sel^p(E/k).
\end{displaymath}
The second inclusion (and therefore also the second inequality) is strict if and only if $\Sha(E/k)[p]$ is non-trivial.

Apart from their role in calculating the Cassels--Tate pairing as explained in~\cite{Fisher-Newton}, the Tate local pairings themselves can be used to study Brauer--Manin obstructions to weak approximation, see~\cite{Ieronymou-Skorobogatov_odd_order} and~\cite{Newton_tBM}.

This paper is structured as follows. Section 2 contains the necessary setup from the literature; section 3 contains the formula from Fisher and Newton and our main result for any odd prime $p$ in Theorem \ref{DeltaPQ_nice}. Indeed, from this point on we will always assume that $p$ is an odd prime number. 

\section{The setup}
In order to make sense of the notation and the statements, we need several constructions and results from~\cite{Explicit-n-descent-I-Algebra}*{Section 3}. This section reviews the ones that are most important for our purposes.

\subsection{The Tate local pairing}
\begin{defn}\label{localpairing_def}
For $v$ a place of $k$ and $K$ any finite extension of $k_v$, one defines a composite pairing 
\small
\begin{displaymath}
(\ ,\ )_K:\HH^1(K,E[p])\times \HH^1(K,E[p])\stackrel{\cup}{\to} \HH^2(K,E[p]\otimes E[p])\stackrel{e_p}{\to} \HH^2(K,\mu_p)\stackrel{\inv_K}{\to}\tfrac1p\bbZ/\bbZ
\end{displaymath}
\normalsize
composed of the cup product $\cup$, the Weil pairing $e_p$ and the Hasse local invariant.
\end{defn}

\begin{remark}{\color{white}This is some filler text.}
\begin{enumerate}
\item Since both the cup product and the Weil pairing are antisymmetric (since $p$ is odd), the pairing $(\ ,\ )_K$ is symmetric.
\item The group $\HH^2(K,\mu_p)$ is a subgroup of $\HH^2(K,\ol{K}^\times)=\Br(K)$, the Brauer group of $K$. The Brauer group of an algebraically closed field is trivial, that of $\bbR$ is of order two; so neither $\Br(\bbC)$ nor $\Br(\bbR)$ have non-trivial $p$-torsion and hence the pairings $(\ ,\ )_\bbC$ and $(\ ,\ )_\bbR$ are trivial and we need only consider finite places $v$.
\end{enumerate}
\end{remark}

The following definition works for any field extension of $k$, but in the remainder, we will only apply it for finite extensions of completions of $k$.

\begin{defn}\label{defR}
Let $K$ be a field extension of $k$. For maps (just of sets) $f:E[p](\ol K)\to \ol K$, a Galois action is defined as follows. For any element $\sigma\in G_K=\Gal(\ol K/K)$ and $P\in E[p](\ol K)$, define $(\sigma f)(P):=\sigma (f(\sigma^{-1}P))$. The set of $G_K$-invariant maps is notated by
\begin{displaymath}
R_K:=\operatorname{Map}_K(E[p](\ol K),\ol K).
\end{displaymath}
Dropping $G_K$-invariance, the notation $\ol{R_K}:=R_K\otimes_K\ol{K}=\operatorname{Map}(E[p](\ol K),\ol K)$ is used.
\end{defn}

From these algebras $R_K$ and $\ol{R_K}$ we will take some specific elements for use in our formula. They are described below.

\begin{prop}\label{H90}~\cite{Explicit-n-descent-I-Algebra}*{p.136}\label{exactpartial}
Define two homomorphisms of groups $w: E[p](\ol K)\to \ol{R_K}^\times$ by $w(S)(-):=e_p(S,-)$ and $\partial: \ol{R_K}^\times\to (\ol{R_K}\otimes_{\ol{K}}\ol{R_K})^\times$ by $(\partial \alpha)(T_1,T_2):=\frac{\alpha(T_1)\alpha(T_2)}{\alpha(T_1+T_2)}$. The sequence
\begin{displaymath}
0\to E[p](\ol K)\stackrel{w}{\longrightarrow}\ol {R_K}^\times\stackrel{\partial}{\longrightarrow}(\ol{R_K}\otimes_{\ol K}\ol{R_K})^\times
\end{displaymath}
is exact. Moreover, we have $H^1(K,\ol{R_K}^\times)=0$.
\end{prop}

This proposition allows to define the following two homomorphisms.

\begin{defn}\label{gammadef} Take any $[\xi]\in \HH^1(K,E[p])$. Then there exists a $\gamma\in \ol{R_K}^\times$ such that
\begin{displaymath}
w(\xi(\sigma))=\sigma(\gamma)/\gamma
\end{displaymath}
holds for all $\sigma\in G_K$. From such $\gamma$, define $\alpha:=\gamma^p$ and $\rho:=\partial\gamma$. Furthermore define homomorphisms 
\begin{align*} 
w_{1,K}:\HH^1(K,E[p])&\to R_K^\times/(R_K^\times)^p,\\
[\xi] &\mapsto \alpha(R_K^\times)^p
\end{align*} 
and 
\begin{align*} 
w_{2,K}:\HH^1(K,E[p])&\to (R_K\otimes_K R_K)^\times/\partial R_K^\times,\\
[\xi] &\mapsto \rho\partial R_K^\times.
\end{align*}
\end{defn}

These $w_{1,K}$ and $w_{2,K}$ are well-defined (e.g. their images are $G_K$-invariant) and injective. Injectivity is given by~\cite{Explicit-n-descent-I-Algebra}*{Lemmas 3.1 and 3.2}.

Finally, one can define the Tate local pairing in the form that we want to have it.

\begin{defn}
For $v$ a finite place of $k$ and $K$ a finite extension of $k_v$, the pairing $[\ ,\ ]_K$ is the pairing on the image of $w_{1,K}$ induced by the pairing $(\ ,\ )_K$ given in Definition \ref{localpairing_def} and it is called the Tate local pairing. For $K=k_v$, we also write $[\ ,\ ]_v$.
\end{defn}

Given any $[\xi]\in \HH^1(K,E[p])$, the $\gamma$ in Definition \ref{gammadef} is not unique, but $\alpha$ and $\rho$ are. Following~\cite{Fisher-Newton}, we use some terminology referring to the elements $\alpha$ and $\rho$ (and sometimes implicitly $\gamma$).

\begin{defn}\label{compatible}
For $[\xi]\in \HH^1(K,E[p])$, we call a pair of elements $\alpha\in R_K^\times$ and $\rho\in (R_K\otimes_K R_K)^\times$ compatible representatives for $[\xi]$ if there exists a $\gamma\in \ol{R_K}^\times$ such that the following hold:
\begin{enumerate}
\item $w(\xi(\sigma))=\sigma\gamma/\gamma$,
\item $\gamma(\calO)=1$,
\item $\gamma^p=\alpha$, and
\item $\partial\gamma=\rho$.
\end{enumerate}
\end{defn}

Notice that the requirement $\gamma(\calO)=1$ is free, since we may always scale $\gamma$ by an element of $\ol{K}^\times$. 

\begin{notation}
Whenever we use the notation $\alpha$, $\rho$ and/or $\gamma$, we will always implicitly assume that they form a set of compatible representatives as above.
\end{notation}

\subsection{The new multiplication}
The advantage of introducing $[\ ,\ ]_K$ on the image of $w_{1,K}$ as opposed to studying the pairing $(\ ,\ )_K$ directly, is that the algebra $R_K$ can be endowed with an alternative multiplication allowing the formulae that will be treated in the next section. In this section we consider the alternative multiplication, writing $\eps_p(S,T)$ for $e_p(S,T)^{1/2}$ for convenience.

For every $T\in E[p](\ol K)$ the set $\ol{R_K}$ contains an indicator function $\delta_T$ given by 
\begin{displaymath}
\delta_T(S):=\begin{cases}1& \text{ if }S=T,\\0 & \text{ otherwise}\end{cases},
\end{displaymath}
which is an element of $R_K$ if and only if $T\in E[p](K)$ holds.

\begin{defn}
For $\rho$ in $(R_K\otimes_K R_K)^\times$ a multiplication $*_\rho$ on $R_K$ is defined by setting
\begin{equation*}
(f*_\rho g)(T)=\sum_{\substack{T_1+T_2=T\\ T_1,\, T_2\in E[p](\ol K)}}\eps_p(T_1,T_2)\rho(T_1,T_2)f(T_1)g(T_2).
\end{equation*}
We write $R_{K,\rho}$ for $(R_K,+,*_\rho)$. The multiplication depends on $\rho$, but we will write $*$ for $*_\rho$ where no confusion is likely to arise. Indeed, $R_{K,\rho}$ is a (non-commutative) ring with multiplicative unit $\delta_\calO$.
\end{defn}

With this alternative multiplication, the indicator functions satisfy some useful properties.

\begin{prop}\label{ind_mult}
For any $T,S\in E[p](\ol K)$ and $\sigma\in \Gal(\ol K/K)$ the following hold:
\begin{enumerate}
\item $\sigma(\delta_T)=\delta_{\sigma T}$,\label{int_mult0}
\item $\delta_T*\delta_S=\eps_p(T,S)\rho(T,S)\delta_{T+S}$,\label{ind_mult1}
\item $\delta_T*\delta_S=e_p(T,S)\delta_S*\delta_T$,
\item $\delta_T^{*p}=\alpha(T)\delta_\calO$,
\item $\delta_T$ is invertible with inverse 
\begin{displaymath}
\delta_T^{-1}=\frac{1}{\gamma(T)\gamma(-T)}\delta_{-T}.
\end{displaymath}
\end{enumerate}
\end{prop}
\begin{proof}
Property \ref{int_mult0} follows by the definition of the Galois action on $R_K$. One proves \ref{ind_mult1} by direct computation; the rest follows.
\end{proof}

\begin{remark}
The alternative multiplication $*_\rho$ can also be defined on $\ol{R_K}$, defining $\ol{R_{K,\rho}}$ similarly. The fact that $R_{K,\rho}$ and hence $\ol{R_{K,\rho}}$ are central simple algebras is important for the proof of Theorem \ref{realthm_non_rat} but this will not be used in the present paper.
\end{remark}

\section{Where the expression appears}
With a symmetric bilinear form with codimension of characteristic different from 2, one can associate a unique quadratic form. We write $q_K$ for the quadratic form associated with $[\ ,\ ]_K$, i.e. $q_K(x)=\frac{1}{2}[x,x]_K$. 

\begin{prop}~\cite{Fisher-Newton}*{Lemma 2.12}\label{extension}
Let $v$ be a finite place of $k$ and let $k_v\subset K\subset L$ be two finite extensions. Then the following relation holds
\begin{displaymath}
q_L=[L:K]q_K.
\end{displaymath}
\end{prop}

\begin{remark}
Proposition \ref{extension} allows us to freely take field extensions of degree coprime to $p$, do our calculations over the bigger field and then divide the outcome by the extension degree. This freedom is the reason why in the previous section we more generally defined everything over field extensions of $k_v$ instead of just over $k_v$ itself.
\end{remark}

\subsection{A suitable field extension}

\begin{prop}\label{PandQ}
For any finite place $v$ of $k$ and any finite extension $K$ of $k_v$, there exists a finite extension $F$ of $K$ that has degree coprime to $p$ such that we have exactly one of two cases:
\begin{enumerate}
\item $E[p](\ol K)=E[p](F)$, or\label{PandQ1}
\item there exist points $P$ and $Q$ that generate $E[p](\ol{K})$ such that $P$ is defined over $F$ and $Q$ is defined over a cyclic field extension $F\subset L$ of degree $p$ and such that the Galois group $\Gal(L/F)=\langle\sigma\rangle$ acts on $E[p](\ol K)$ by $\sigma(P)=P$ and $\sigma(Q)=Q+P$.\label{PandQ2}
\end{enumerate}
\end{prop}
\begin{proof}
By the action of $G_K=\Gal(\ol K/K)$ on $E[p](\ol K)$ we get a homomorphism 
\begin{equation*}
G_K\to \Aut(E[p](\ol K))\cong \GL_2(\bbF_p)
\end{equation*}
where we consider automorphisms of groups. The isomorphism is by choosing a basis. Let $L$ be the fixed field of the kernel of this map. Then $E[p](L)=E[p](\ol K)$ holds and we have maps
\begin{equation*} 
G_K\longrightarrow \Gal(L/K)\hookrightarrow \GL_2(\bbF_p).
\end{equation*}
Let $C\subset \Gal(L/K)$ be a Sylow-$p$-subgroup and let $F=L^C$ be the field fixed by $C$. Then we have $K\subseteq F\subseteq L$ and since $\#\GL_2(\bbF_p)=p(p-1)(p^2-1)$ is divisible by only a single factor of $p$ and $C$ is (isomorphic to) a subgroup of $\GL_2(\bbF_p)$, we have two cases: either $C$ is trivial or $C$ is cyclic of order $p$. In either case we have $p\nmid [F:K]$, so $F$ is a candidate field for this proposition.

In the first case we have $F=L$. This yields case \ref{PandQ1} from the statement.

In the second case we have $\Gal(L/F)\cong \bbZ/p\bbZ$. It is a basic fact that all Sylow-$p$-subgroups of a finite group are conjugates and we know that the subgroup
\begin{equation*} 
H=\left\langle \left(\begin{array}{cc} 1&1\\ 0&1\end{array}\right)\right\rangle\subset \GL_2(\bbF_p)
\end{equation*}
is a Sylow-$p$-subgroup. By application of an inner automorphism of $\GL_2(\bbF_p)$ and interpreting $\Gal(L/K)$ as a subgroup of $\GL_2(\bbF_p)$, we can thus identify $C$ with $H$.

Let $P\in E[p](L)$ correspond to the vector $(1,0)^t$ and $Q\in E[p](L)$ to the vector $(0,1)^t$. Then the element $\sigma\in \Gal(L/F)$ that corresponds to the given generator of $H$ yields $\sigma(P)=P$ and $\sigma(Q)=Q+P$. 
\end{proof}

\begin{assumption}\label{assumption_K}
From now on, we will always assume that our field $K$ is replaced by a degree coprime to $p$ field extension such that at least one non-trivial point of $E[p](\ol K)$ is defined over $K$ and name an arbitrary such point $P$. 
\end{assumption}

\begin{defn}
For $Q\in E[p](\ol K)$ chosen such that $P$ and $Q$ generate $E[p](\ol K)$, we write 
\begin{displaymath}
\Delta_{P,Q}=\delta_Q+\delta_{Q+P}+\ldots+\delta_{Q+(p-1)P}\in \ol{R_K}.
\end{displaymath}
\end{defn}

\begin{lemma}
For any $Q\in E[p](\ol K)$ that together with $P$ generates $E[p](\ol K)$, the element $\Delta_{P,Q}$ lies in $R_K$.
\end{lemma}
\begin{proof}
The statement is obvious in situation \ref{PandQ1} of Proposition \ref{PandQ}. For any $Q\in E[p](L)$ as in situation \ref{PandQ2}, one applies $\sigma$ to the expression for $\Delta_{P,Q}$. The only remaining case is $E[p](\ol K)=E[p](L)$ as in situation \ref{PandQ2} but with $\sigma(Q)\neq Q+P$. 

Write $\sigma(Q)=mQ+nP$ for some $m,n\in \bbZ/p\bbZ$. Then we find
\begin{displaymath}
Q=\sigma^p(Q)=m^pQ+\sum_{i=0}^{p-1}m^i nP = mQ+\sum_{i=0}^{p-1}m^i nP
\end{displaymath}
and hence $m=1$. From $Q\notin E[p](K)$, we conclude $n\neq 0$ and hence $\gcd(n,p)=1$. Therefore the orbit of $Q$ is $\{Q,Q+P,\ldots,Q+(p-1)P\}$ and $\Delta_{P,Q}$ is invariant under $\Gal(L/K)$.
\end{proof}

\subsection{The results}
The content of the following theorem is functionally identical to that of~\cite{Fisher-Newton}*{Theorem 2.13}.

\begin{thm}\label{realthm_non_rat}
Let $P\in E[p](K)$ and $Q\in E[p](\ol K)$ be such that $P$ and $Q$ generate $E[p](\ol K)$ as an abelian group. For $\alpha\in w_{1,K}(\HH^1(K,E[p]))$, let $\rho\in(R_K\otimes_K R_K)^\times$ together with $\alpha$ form a pair of compatible representatives for some $[\xi]\in \HH^1(K,E[p])$. Write $\{\ ,\ \}_K$ for the Hilbert norm residue symbol of $K$ and $\iota_{P,Q}:\mu_p(L)\to \tfrac{1}{p}\bbZ/\bbZ$ for the group homomorphism determined by $e_p(Q,P)\mapsto \tfrac{1}{p}$. Then we have
\begin{equation*}
q_K(\alpha)= \begin{cases}\iota_{P,Q}\{\alpha(P),\Delta_{P,Q}^{*p}\}_K & \text{if}\ \Delta_{P,Q}^{*p}\neq 0,\\ 0 & \text{otherwise}.\end{cases}
\end{equation*}
\end{thm}

It remains to find an expression for $\Delta_{P,Q}^{*p}$ as ~\cite{Fisher-Newton}*{Theorem 2.1} does only for the case $p=3$.

\begin{thm}\label{DeltaPQ_nice}
In the setting as in the previous theorem, we have
\begin{displaymath}
\Delta_{P,Q}^{*p} = \sum_{\substack{i_1,i_2,\ldots,i_p\in\bbZ/p\bbZ\\ p|\sum_\ell i_\ell}}e_p(Q,P)^{\sum_{\ell=1}^p\ell\cdot i_\ell}\prod_{\ell=1}^p\gamma(Q+i_\ell P)\delta_\calO.
\end{displaymath}
\end{thm}
\begin{proof}
This can be done by direct calculation starting with determining $\Delta_{P,Q}^{*2}$, proceeding by induction, writing $\rho$ in terms of $\gamma$ and cancelling factors. Doing so, one arrives at an expression 
\begin{equation}\label{DeltaPQ_not_nice}
\Delta_{P,Q}^{*p} = \sum_{i_1,i_2,\ldots,i_p\in\bbZ/p\bbZ}\eps_p(Q,P)^{\sum_{\ell=1}^p(2\ell-1)i_\ell}\frac{\prod_{\ell=1}^p\gamma(Q+i_\ell P)}{\gamma((\sum_{\ell=1}^p i_\ell) P)}\delta_{(\sum_{\ell=1}^p i_\ell)P}
\end{equation}
which we need to show is equal to the expression given in the statement of the theorem. Firstly, it is straightforward to see that the $\delta_\calO$-terms agree from $\eps_p(Q,P)^2=e_p(Q,P)$ and $e_p(Q,P)^p=1$; the other terms cancel since we already know from the statement of Theorem \ref{realthm_non_rat} that $\Delta_{P,Q}^{*p}$ gives an element of $K\delta_\calO$.
\end{proof}

\begin{remark}
Specifying to the case $p=3$ and comparing the result from Theorem \ref{DeltaPQ_nice} to the expression in~\cite{Fisher-Newton}*{Proposition 2.14}, one finds that no further terms cancel in general. In this sense, our expression is the best possible. 
\end{remark}

On top of page 902 of~\cite{Explicit-n-descent-III-Algorithms} the authors explain that given $\alpha$, in practice it is often easy to determine $\rho$ such that the pair form compatible representatives, in such a way that determination of $\gamma$ is skipped. For this reason a formula in terms of just $\rho$ may be of interest. It is given in the theorem below. The computations in its proof are completely analogous to -- in fact contained in -- those in the proof of Theorem \ref{DeltaPQ_nice}.

\begin{thm}
In the setting as in the previous two theorems, we have
\small
\begin{displaymath}
\Delta_{P,Q}^{*p} = \sum_{\substack{i_1,i_2,\ldots,i_p\in\bbZ/p\bbZ\\ p|\sum_\ell i_\ell}}e_p(Q,P)^{\sum_{\ell=1}^p\ell\cdot i_\ell}\left(\prod_{j=1}^{p-1}\rho\left(jQ+\sum_{\ell=1}^j i_\ell P,Q+i_{j+1}P\right)\right)\delta_\calO.
\end{displaymath}
\end{thm}

\normalsize
\subsection{About the norm}
At the end of their section 2, Fisher and Newton remark that in the situation where all $p$-torsion is defined over the base field one may choose $\gamma$ to satisfy the equality $\gamma(aP+bQ)=\gamma(P)^a\gamma(Q)^b$ for all $P,Q\in E[p](K)$ and all $a,b\in\{0,\ldots,p-1\}$, and they notice that in this way, their explicit formula for $p=3$ recovers the one given earlier by O'Neil~\cite{ONeil_Period-Index}*{Proposition 3.4} by the fact that the Hilbert norm residue symbol $\{a,b\}_K$ is trivial if $b$ is a norm in $K(a^{1/p})$. 

\begin{thm}\label{norm}
In the situation $E[p](\ol K)=E[p](K)$, choose $\gamma$ as described above and let $P,Q\in E[p](K)$ be generators. Write $M=K(\gamma(P))$. Then the following two equalities hold:
\begin{enumerate}
\item $\Norm_{M/K}\left(1+\gamma(P)+\ldots+\gamma(P)^{p-1}\right) =(1-\alpha(P))^{p-1}$, and \label{norm1}
\item $\Delta_{P,Q}^{*p} = \alpha(Q)(1-\alpha(P))^{p-1}.$\label{norm2}
\end{enumerate}
\end{thm}
\begin{proof}
We first prove \ref{norm1}. We have 
\begin{displaymath}
\Bigl(1+\gamma(P)+\cdots+\gamma^{p-1}(P)\Bigr)(1-\gamma(P)) = 1-\gamma^p(P) = 1-\alpha(P)
\end{displaymath}
and since the norm is multiplicative:
\begin{equation*}
\Norm_{M/K}\Bigl(1+\gamma(P)+\cdots+\gamma^{p-1}(P)\Bigr)=\frac{\Norm_{M/K}(1-\alpha(P))}{\Norm_{M/K}(1-\gamma(P))}.
\end{equation*} 
As $1-\alpha(P)$ is an element of $K$, its norm is just $(1-\alpha(P))^p$. The norm of $1-\gamma(P)$ is easily calculated from the equality 
\begin{equation*}
(X-\gamma(P))(X-\zeta_p\gamma(P))\cdots(X-\zeta_p^{p-1}\gamma(P))=X^p-\gamma^p(P)=X^p-\alpha(P)
\end{equation*}
where $\zeta_p$ is a primitive $p$th root of unity and substituting $X=1$. 

To prove the second statement, we compute in the non-commutative ring $\bbQ(\zeta_p)\langle X,Y\rangle/(YX-\zeta_p^2XY)$ which is a subring of $R_K$ after identification $X\mapsto \delta_P$, $Y\mapsto \delta_Q$ and $\zeta_p\mapsto \eps_p(Q,P)$. By our choice of $\gamma$, we may simplify the multiplication by using $\rho(Q,P)=\gamma(Q)\gamma(P)/\gamma(Q+P)=1$ and thus we need to prove the equality
\begin{displaymath}
\left(\sum_{i=0}^{p-1} YX^i\zeta_p^{-i}\right)^p = Y^p(1-X^p)^{p-1}.
\end{displaymath}
We remark that for every $j\in \{0,\ldots,p-1\}$ we have
\begin{displaymath}
X^jY=\zeta_p^{-2j}YX^j
\end{displaymath}
and hence
\begin{align*}
\left(\sum_{i=0}^{p-1} YX^i\zeta_p^{-i}\right)^p &=\left(Y\sum_{i=0}^{p-1}X^i\zeta_p^{-i}\right)^p\\
&= Y^p\prod_{j=0}^{p-1}\left(\sum_{i=0}^{p-1}X^i\zeta_p^{-(2j+1)i}\right)\\
&= Y^p\prod_{j=0}^{p-1}\left(\frac{1-X^p\left(\zeta_p^{-(2j+1)}\right)^p}{1-X\zeta_p^{-(2j+1)}}\right)\\
&= Y^p(1-X^p)^p\prod_{j=0}^{p-1}\left(1-X\zeta_p^{-(2j+1)}\right)^{-1}.
\end{align*}
We finish the argument by remarking that $2j+1$ runs through all of $\bbZ/p\bbZ$ when $j$ does since $\gcd(2,p)=1$ holds. Therefore the last product equals
\begin{displaymath}
\prod_{j=0}^{p-1}\left(1-X\zeta_p^j\right)^{-1} = \left(1-X^p\right)^{-1}
\end{displaymath}
and we have arrived at the desired equality.
\end{proof}

Combining the two parts of this theorem shows us that in the situation described above, just like the case $p=3$, the expressions from Theorems \ref{realthm_non_rat} and \ref{DeltaPQ_nice} recover the formula from~\cite{ONeil_Period-Index}*{Proposition 3.4}.

\begin{remark}
At the time that the author wrote his master's thesis~\cite{mscthesis}, little effort had been made to find a complete proof of part \ref{norm2} of Theorem \ref{norm}; the statement is contained in the thesis as a conjecture.
\end{remark}


\end{document}